\documentclass[12pt, a4paper]{article}
\usepackage{graphicx}
\usepackage[utf8]{inputenc}
\usepackage{rotating}
\usepackage{amsmath}
\usepackage{amsthm}
\usepackage{amsfonts}
\usepackage[T1]{fontenc}
\usepackage{anysize}
\marginsize{2cm}{1.5cm}{2cm}{1.5cm}
	\newtheorem{theorem}{Theorem}[section]

	\newtheorem{conc}[theorem]{Conclusion}

	\newtheorem{coro}[theorem]{Corollary}

	\newtheorem{rem}[theorem]{Remark}

\title{\Large \textbf{A new generalization of beta function with three parameters Mittag-Leffler function}}
\author{Muhammet AY \\ 
\small Ahi Evran University, Department of Mathematics, Kırşehir, Turkey. \\ \small 161017007@ogrmail.ahievran.edu.tr}
\date{}

\begin{document}
\maketitle
\begin{abstract} The main object of this paper is to present a new generalized beta function which defined by three parametres Mittag-Leffler function. We also introduce new generalizations of hypergeometric and confluent hypergeometric functions with the help of new generalized beta function. Furthermore, we obtained various properties of these functions such as integral representations, Mellin transforms, differentiation formulas, transformation formulas, recurrence relations and summation formula. 
\end{abstract}

\section{Introduction}
We first remember the  gamma and the beta functions \cite{1} which defined by
 \begin{equation} 
 \Gamma(x)=\int_{0}^{\infty} t^{x-1} e^{-t}dt,\ Re(x)>0,
 \end{equation}
 \begin{equation}\label{1}
 B(x,y) =  \int_{0}^{1}t^{x-1}(1-t)^{y-1}dt,\ Re(x)>0,Re(y)>0.
 \end{equation}
The integral representation of classical Gauss hypergeometric function (GHF), which defined as
\begin{equation} 
{}_{2}\textrm{F}_{1}(a,b;c;z)= \sum_{n=0}^{\infty} \frac{(a)_n (b)_n}{(c)_n} \frac{z^n}{n!}
\end{equation}
is as follows \cite{1}
\begin{eqnarray} 
{}_{2}\textrm{F}_{1}(a,b;c;z)= \frac{1}{B(b, c-b)} \int_{0}^{1} t^{b-1} (1-t)^{c-b-1}(1-zt)^{-a} dt 
\end{eqnarray}
where $|arg(1-z)|<\pi,\ Re(c)>Re(b)>0$. The GHF can also be written in terms of beta function \eqref{1} as follows \cite{1}:
\begin{align} 
{}_{2}\textrm{F}_{1}(a,b;c;&z)= \sum_{n=0}^{\infty} (a)_n \frac{B(b+n,\;c-b)}{B(b,\;c-b)}  \frac{z^n}{n!}.& \\ \nonumber &(|z|<1,\ Re(c)>Re(b)>0 )&
\end{align} 

The integral representation of classical confluent hypergeometric function (CHF), which defined as
\begin{eqnarray} 
{}_{1}\textrm{F}_{1}(b;c;z)=\Phi(b;c;z)= \sum_{n=0}^{\infty} \frac{(b)_n}{(c)_n} \frac{z^n}{n!}
\end{eqnarray}
is as follows \cite{1}
\begin{align} 
\Phi(b;c;z)= \frac{1}{B(b,c-b)} \int_{0}^{1} t^{b-1} (1-t)^{c-b-1}e^{zt} dt
\end{align}
where $(Re(c)>Re(b)>0)$. The CHF can also be written in terms of beta function \eqref{1} as follows \cite{1}:
\begin{align} 
\Phi(b;c;z)= &\sum_{n=0}^{\infty} \frac{B(b+n,\;c-b)}{B(b,\;c-b)} \frac{z^n}{n!}.&  \\ \nonumber
&(Re(c)>Re(b)>0)&
\end{align}

In recent years, several extensions of well known special functions have been thought by several authors [\ref{k2}-\ref{k6},\ref{k10}]. In 1994, Chaudry and Zubair \cite{2} have introduced the following extension of gamma function
\begin{eqnarray} \label{25}
\Gamma_p (x)= \int_{0}^{\infty} t^{x-1} \exp{\left(-t-\frac{p}{t}\right)}dt, \ Re(p)>0.
\end{eqnarray}
In 1997, Chaudry et al. \cite{3} studied the following extension of Euler's beta function
\begin{align} \label{26}
B_p(x,y)=\int_{0}^{1} t^{x-1}(1-t)^{y-1}exp\left[-\frac{p}{t(1-t)}\right] dt, \  Re(p)>0.
\end{align} 
Afterwards, Chaudry et al. \cite{7} used $B_p(x,y)$ to extended the hypergeometric and confluent hypergeometric functions as follows: 
\begin{align} \label{27}
{}_{}\textrm{F}_{p}(a&,b;c;z)= \sum_{n=0}^{\infty} (a)_n \frac{B_{p}(b+n,\;c-b)}{B(b,\;c-b)}  \frac{z^n}{n!},& \\ \nonumber   &( p\geq0;|z|<1\ Re(c)>Re(b)>0)& 
\end{align} 
\begin{align} \label{28}
{}_{}{\Phi}_{p}(b;&c;z)= \sum_{n=0}^{\infty}\frac{B_{p}(b+n,\;c-b)}{B(b,\;c-b)}\frac{z^n}{n!}& \\ \nonumber &( p\geq0;\ Re(c)>Re(b)>0 )&
\end{align} 
and give the Euler integral representations
\begin{align}  \nonumber
{}_{}\textrm{F}_{p}(a,b;c&;z)= \frac{1}{B(b, c-b)} \int_{0}^{1} t^{b-1} (1-t)^{c-b-1}(1-zt)^{a}exp\left[-\frac{p}{t(1-t)}\right] dt,& \\  &(p>0 ;\ p=0 \ and \ |arg(1-z)|<\pi ;\ Re(c)>Re(b)>0)&
\end{align} 
\begin{align}  \nonumber
{}_{}{\Phi}_{p}(b;c;z)=& \frac{1}{B(b, c-b)} \int_{0}^{1} t^{b-1} (1-t)^{c-b-1}exp\left[zt-\frac{p}{t(1-t)}\right] dt.& \\ &(p>0 ;\ p=0\ and \ Re(c)>Re(b)>0)&
\end{align} 

They also obtained the integral representations, differentiation properties, Mellin transforms, transformation formulas, recurence relations, summation and asymptotic formulas for these functions \cite{7}.

In this paper, we use there parameters Mittag-Leffler function, which defined in Prabhakar \cite{11} as
\begin{eqnarray} \label{2}
	E_{\alpha, \beta}^{\gamma}(z)= \sum_{n=0}^{\infty} \frac{(\gamma)_n}{\Gamma(\alpha n + \beta)}\frac{z^n}{n!},\ Re(\alpha)>0
\end{eqnarray} 
where $\alpha$, $\beta$, $\gamma$ $\in$ $\mathbb{C}$. It is obvious that,
\begin{align} \nonumber
E_{\alpha, \beta}^{1}(z) &=\sum_{n=0}^{\infty} \frac{(1)_n}{\Gamma(\alpha n + \beta)}\frac{z^n}{n!}= \sum_{n=0}^{\infty} \frac{z^n}{\Gamma(\alpha n + \beta)}=E_{\alpha, \beta}(z),&  \\ \nonumber
E_{\alpha, 1}^{1}(z)&= \sum_{n=0}^{\infty} \frac{(1)_n}{\Gamma(\alpha n + 1)}\frac{z^n}{n!}= \sum_{n=0}^{\infty} \frac{z^n}{\Gamma(\alpha n + 1)}= E_{\alpha}(z),& \\ \nonumber
E_{1, 1}^{1}(z) &=\sum_{n=0}^{\infty} \frac{(1)_n}{\Gamma( n + 1)}\frac{z^n}{n!}=\sum_{n=0}^{\infty} \frac{z^n}{n!}=e^z. &
\end{align} 
For more properties of the generalized Mittag-Leffler function, we refer to the reader \cite{9}.

\section{New generalizations of gamma and beta functions}
In this section, we present a new generalizations of gamma and beta functions by using  there parameters Mittag-Leffler function in \eqref{2} as
\begin{align} \label{3}
{}^{M\!L\!}\Gamma_{p}(x) =\int_{0}^{\infty} t^{x-1}E_{\alpha, \beta}^{\gamma}\left(-t-\frac{p}{t}\right) dt \\ \nonumber	(Re(p)>0, Re(x)>0, Re(\alpha)>0)
\end{align} 
and
\begin{align} \label{4} 
{}^{M\!L\!\!}{B}_{p}(x,y)=\int_{0}^{1}t^{x-1}(1-t)^{y-1}E_{\alpha, \beta}^{\gamma}\left(-\frac{p}{t(1-t)}\right)dt. \\ \nonumber (Re(p)>0, Re(x)>0, Re(y)>0, Re(\alpha)>0) 
\end{align} 
We call these functions as ML-generalized gamma and ML-generalized beta functions, respectively.

Clearly, when $\alpha=\beta=\gamma=1$, equations \eqref{3} and \eqref{4} recudes to \eqref{25} and \eqref{26}. Also, when $p=0$ and $\beta=1$ (or $\beta=2$)  , recudes to classical gamma and beta functions \cite{1}. 

\begin{theorem} For the product of ML-generalized gamma function, we have the following integral representation:
\begin{align} \label{5} \nonumber
{}^{M\!L\!}\Gamma_{p}(x)&{}^{M\!L\!}\Gamma_{p}(y)=4\int_{0}^{\frac{\pi}{2}}\int_{0}^{\infty} r^{2(x+y)-1} \cos^{2x-1}\theta \sin^{2y-1}\theta& \\  &\times E_{\alpha, \beta}^{\gamma}\left(-r^2\cos^2\theta -\frac{p}{r^2\cos^2\theta}\right) E_{\alpha, \beta}^{\gamma}\left(-r^2\sin^2\theta -\frac{p}{r^2\sin^2\theta}\right)drd\theta.&
\end{align}
\end{theorem}
\begin{proof}
Substituting $t=\eta^2$ in (\ref{5}), we get
\begin{align*} 
{}^{M\!L\!}\Gamma_{p}(x) =2\int_{0}^{\infty} \eta^{2x-1}E_{\alpha, \beta}^{\gamma}\left(-\eta^2-\frac{p}{\eta^2}\right) d\eta.
\end{align*}
Therefore 
\begin{align*} \nonumber
{}^{M\!L\!}\Gamma_{p}(x){}^{M\!L\!}\Gamma_{p}(y) =4\int_{0}^{\infty}\int_{0}^{\infty} \eta^{2x-1}\xi^{2y-1} E_{\alpha, \beta}^{\gamma}\left(-\eta^2-\frac{p}{\eta^2}\right)E_{\alpha, \beta}^{\gamma}\left(-\xi^2-\frac{p}{\xi^2}\right) d\eta d\xi.
\end{align*}
Letting $\eta=r\cos\theta$ and $\xi=r\sin\theta$ in the above equality,
\begin{align*} \nonumber
{}^{M\!L\!}\Gamma_{p}(x)&{}^{M\!L\!}\Gamma_{p}(y)=4\int_{0}^{\frac{\pi}{2}}\int_{0}^{\infty} r^{2(x+y)-1} \cos^{2x-1}\theta \sin^{2y-1}\theta& \\ \nonumber &\times E_{\alpha, \beta}^{\gamma}\left(-r^2\cos^2\theta -\frac{p}{r^2\cos^2\theta}\right)  E_{\alpha, \beta}^{\gamma}\left(-r^2\sin^2\theta -\frac{p}{r^2\sin^2\theta}\right)drd\theta& 
\end{align*}
which completes the proof.
\end{proof}
\begin{rem}
Putting $p=0$ and $\beta=1$ (or $\beta=2$)  in \eqref{5}, we get the classical relation between the gamma and beta functions:
\begin{align*}
B(x,y)=\frac{\Gamma(x)\Gamma(y)}{\Gamma(x+y)}.
\end{align*}
\end{rem}
\begin{theorem}	For the ML-generalized beta function, we have the following summation relation:
	\begin{eqnarray} 
	{}^{M\!L\!\!}{B}_{p}(x,1-y) = \sum_{n=0}^{\infty} \frac{(y)_n}{n!}	{}^{M\!L\!\!}{B}_{p}(x+n,1), \ Re(p)>0.	 
	\end{eqnarray} 
\end{theorem}
\begin{proof}
From the definition of the ML-generalized beta function \eqref{4}, we have 
 \begin{align*}
{}^{M\!L\!\!}{B}_{p}(x,1-y)=\int_{0}^{1}t^{x-1}(1-t)^{-y}	E_{\alpha, \beta}^{\gamma}\left(-\frac{p}{t(1-t)}\right)dt.	
\end{align*}
Using the following series expansion
\begin{align*} 
(1-t)^{-y}=\sum_{n=0}^{\infty}(y)_n\frac{t^n}{n!}, \ |t|<1 
\end{align*}
we obtain
\begin{align*}
{}^{M\!L\!\!}{B}_{p}(x,1-y)=\int_{0}^{1}t^{x-1}\left[ \sum_{n=0}^{\infty} \frac{(y)_n}{n!} t^n\right]	E_{\alpha, \beta}^{\gamma}\left(-\frac{p}{t(1-t)}\right) dt. 
\end{align*}
Therefore, interchanging the order of integration and summation and then using definition \eqref{4}, we obtain
	\begin{align*} 
	{}^{M\!L\!\!}{B}_{p}(x,1-y) &= \sum_{n=0}^{\infty} \frac{(y)_n}{n!} \int_{0}^{1}t^{x+n-1}E_{\alpha, \beta}^{\gamma}\left(-\frac{p}{t(1-t)}\right) dt & \\ 
	&=\sum_{n=0}^{\infty} \frac{(y)_n}{n!}	{}^{M\!L\!\!}{B}_{p}(x+n,1)&
	\end{align*}
which completes the proof.
\end{proof}
\begin{theorem} For the ML-generalized beta function, we have the following functional relation:
	\begin{eqnarray}
	{}^{M\!L\!\!}{B}_{p}(x,y) =  {}^{M\!L\!\!}{B}_{p}(x,y+1) +{}^{M\!L\!\!}{B}_{p}(x+1,y), \ Re(p)>0.	
	\end{eqnarray} 
\end{theorem} 
\begin{proof}
The left-hand side of \eqref{4} equals
\begin{align*}\nonumber
{}^{M\!L\!\!}{B}_{p}(x,y)&= \int_{0}^{1} t^{x-1}(1-t)^{y-1}	E_{\alpha, \beta}^{\gamma}\left(-\frac{p}{t(1-t)}\right)dt& \\ \nonumber
&=\int_{0}^{1}[t+(1-t)]\left[ t^{x}(1-t)^{y-1}+t^{x-1}(1-t)^{y}\right]	E_{\alpha, \beta}^{\gamma}\left(-\frac{p}{t(1-t)}\right) dt& \\ \nonumber&=  \int_{0}^{1}t^{x}(1-t)^{y-1}	E_{\alpha, \beta}^{\gamma}\left(-\frac{p}{t(1-t)}\right)dt &\\ &+ \nonumber\int_{0}^{1}t^{x-1}(1-t)^{y}	E_{\alpha, \beta}^{\gamma}\left(-\frac{p}{t(1-t)}\right)dt& \\ 
&= {}^{M\!L\!\!}{B}_{p}(x+1,y) + {}^{M\!L\!\!}{B}_{p}(x,y+1) & 
\end{align*} which completes the proof.
\end{proof} 
\begin{theorem}	For the ML-generalized beta function, we have the following summation relation:
	\begin{eqnarray} \label{6}
{}^{M\!L\!\!}{B}_{p}(x,y) = \sum_{n=0}^{\infty} {}^{M\!L\!\!}{B}_{p}(x+n,y+1), \ Re(p)>0.
	\end{eqnarray} 
\end{theorem}  
\begin{proof}
From the definition of the ML-generalized beta function \eqref{4}, we have
 \begin{align*} 
{}^{M\!L\!\!}{B}_{p}(x,y)=\int_{0}^{1}t^{x-1}(1-t)^{y-1}	E_{\alpha, \beta}^{\gamma}\left(-\frac{p}{t(1-t)}\right)dt.
\end{align*}
Using  $(1-t)^{y-1}=(1-t)^{y}\sum_{n=0}^{\infty}{t^n}, \; \;  (|t|<1)$, we get 
\begin{align*} {}^{M\!L\!\!}{B}_{p}(x,y)=\int_{0}^{1}t^{x-1}(1-t)^{y}\left[ \sum_{n=0}^{\infty} t^n\right]E_{\alpha, \beta}^{\gamma}\left(-\frac{p}{t(1-t)}\right)dt. 
\end{align*}
Interchanging the order of integration and summation, we get
\begin{align*} \nonumber
{}^{M\!L\!\!}{B}_{p}(x,y) &=  \sum_{n=0}^{\infty} \int_{0}^{1}t^{x+n-1}(1-t)^{y}E_{\alpha, \beta}^{\gamma}\left(-\frac{p}{t(1-t)}\right) dt& \\ 
&=\sum_{n=0}^{\infty}{}^{M\!L\!\!}{B}_{p}(x+n,y+1)&
\end{align*}
which completes the proof.
\end{proof} 
The ML-generalized beta function is represented by the various following integral representations:
\begin{theorem} Let $Re(p)>0,\ Re(x)>0,\ Re(y)>0,\ Re(\alpha)>0$, then we have the following integral representations for ML-generalized beta function:
	\begin{align} 
	\label{7} {}^{M\!L\!\!}{B}_{p}(x,y)&=2\int_{0}^{\frac{\pi}{2}} \cos^{2x-1}\theta \sin^{2y-1}\theta E_{\alpha, \beta}^{\gamma}(-p\sec^{2}\theta\csc^{2}\theta)d\theta,&  \\\label{8} 
	{}^{M\!L\!\!}{B}_{p}(x,y)&=\int_{0}^{\infty} \frac{u^{x-1}}{(1+u)^{x+y}} E_{\alpha, \beta}^{\gamma}\left[-p\left(2+u+\frac{1}{u}\right)\right]du,& \\ \label{9} 
	{}^{M\!L\!\!}{B}_{p}(x,y)&=2^{1-x-y}\int_{-1}^{1} (1+u)^{x-1} (1-u)^{y-1}E_{\alpha, \beta}^{\gamma}\left(-\frac{4p}{(1-u^2)}\right)du,& \\ \label{10}
	{}^{M\!L\!\!}{B}_{p}(x,y)&=(c-a)^{1-x-y}\int_{a}^{c} (u-a)^{x-1} (c-u)^{y-1} E_{\alpha, \beta}^{\gamma}\left(-\frac{p(c-a)^2}{(u-a)(c-u)}\right)du.& 
	\end{align}
\end{theorem}
\begin{proof}
	Equations \eqref{7},\eqref{8}, \eqref{9} and \eqref{10} can be obtain by using the transformations $t=\cos^2\theta$, $t=\frac{u}{1+u}$, $t=\frac{1+u}{2}$ and $t=\frac{u-a}{c-a}$ in equation \eqref{4} respectively.
\end{proof}
\begin{theorem} Mellin transform of the ML-generalized beta function is given by 
\begin{align} \label{11}
\mathfrak{M}\left[{}^{M\!L\!\!}{B}_{p}(x,y) : s\right]={}^{M\!L\!}\Gamma(s)B(x+s,y+s).
\end{align} 
$$ (Re(s)>0,\ Re(p)>0,\ Re(x+ s)>0,\ Re(y+ s)>0,\ Re(\alpha)>0) $$
\end{theorem} 
\begin{proof}
Taking Mellin transform of ML-generalized beta function, we get 
\begin{align*}	
\mathfrak{M}\left[ {}^{M\!L\!\!}{B}_{p}(x,y) : s\right]=\int_{0}^{\infty} p^{s-1} \left[\int_{0}^{1} t^{x-1}(1-t)^{y-1}E_{\alpha, \beta}^{\gamma}\left(-\frac{p}{t(1-t)}\right)dt\right]dp. 
\end{align*}
From the uniform convergence of the integral, the order of integration can be interchanged. Therefore, we have 
\begin{align*} 		
\mathfrak{M}\left[ {}^{M\!L\!\!}{B}_{p}(x,y) : s\right]=   \int_{0}^{1}t^{x-1}(1-t)^{y-1}\left[\int_{0}^{\infty} p^{s-1}E_{\alpha, \beta}^{\gamma}\left(-\frac{p}{t(1-t)}\right)dp\right]dt. 	
\end{align*}
Now using the one-to-one transformation $u=\frac{p}{t(1-t)}$, we get 
 \begin{align*}  \nonumber
\mathfrak{M}\left[ {}^{M\!L\!\!}{B}_{p}(x,y) : s\right]&= \int_{0}^{1}t^{x+s-1}(1-t)^{y+s-1}dt \times \int_{0}^{\infty} u^{s-1}E_{\alpha, \beta}^{\gamma}\left(-u\right) du& \\ \nonumber &={}^{M\!L\!}\Gamma(s)\int_{0}^{1}t^{x+s-1}(1-t)^{y+s-1}dt &\\
&= {}^{M\!L\!}\Gamma(s)B(x+s,y+s)&	
\end{align*}
which completes the proof. 
\end{proof} 
\begin{coro}By the Mellin inversion formula, we have the following complex integral representation for ML-generalized beta function:
\begin{eqnarray}
 {}^{M\!L\!\!}{B}_{p}(x,y)=\frac{1}{2\pi i} \int_{-i\infty}^{+i \infty} {}^{M\!L\!}\Gamma(s)B(x+s,y+s)p^{-s}ds
\end{eqnarray}
\end{coro}
\begin{proof}
Taking inversion of \eqref{11}, we get the result.
\end{proof}
\section{ML-generalizations of hypergeometric functions}
In this section, using the ML-generalized beta function \eqref{4} we define the ML-generalizations of Gauss hypergeometric and confluent hypergeometric functions as 
\begin{align} \label{12}
&{}^{M\!L\!}\textrm{F}_{p}(a,b;c;z)= \sum_{n=0}^{\infty} (a)_n \frac{{}^{M\!L\!\!}B_{p}(b+n,\;c-b)}{B(b,\;c-b)} \frac{z^n}{n!}& \\ \nonumber  
&(p\geq0; \ |z|<1;\ Re(c)>Re(b)>0,\ Re(\alpha)>0)&
\end{align}   
and
\begin{align} \label{13}
{}^{M\!L\!}{\Phi}_{p}(b;c;z)= \sum_{n=0}^{\infty}\frac{{}^{M\!L\!\!}B_{p}(b+n,\;c-b)}{B(b,\;c-b)}  \frac{z^n}{n!} \\ \nonumber
(p\geq0;\ Re(c)>Re(b)>0,\ Re(\alpha)>0)
\end{align} 
respectively. 

Clearly, when $\alpha=\beta=\gamma=1$, equations \eqref{12} and \eqref{13} recudes to \eqref{27} and \eqref{28}. Also, when $p=0$ and $\beta=1$ (or $\beta=2$) , recudes to classical hypergeometric and confluent hypergeometric functions \cite{1}. 
\begin{theorem}
Let $p>0;\ p=0,\ |\arg(1-z)|<\pi;\ Re(c)>Re(b)>0,\ Re(\alpha)>0$, then we have the following integral representations for ML-generalized GHF:
\begin{align} \label{14} \nonumber
{}^{M\!L\!\!}\textrm{F}_{p}(&a,b;c;z)=\frac{1}{B(b,c-b)} \int_{0}^{1}t^{b-1}(1-t)^{c-b-1}(1-zt)^{-a}	E_{\alpha, \beta}^{\gamma}\left(-\frac{p}{t(1-t)}\right)dt,
\end{align}
\begin{align}  \nonumber 
{}^{M\!L\!\!}\textrm{F}_{p}(a,b;c;z)= \frac{1}{B(b,c-b)} \int_{0}^{\infty}u^{b-1}&(1+u)^{a-c}(1+u(1-z))^{-a}E_{\alpha, \beta}^{\gamma}\left(-p\left(2+u+\frac{1}{u}\right)\right)du,& 
\end{align}
\begin{align} \nonumber
 {}^{M\!L\!\!}\textrm{F}_{p}(a,b;c;z)= \frac{2}{B(b,c-b)} \int_{0}^{\frac{\pi}{2}}\sin^{2b-1}v&\cos^{2c-2b-1}v(1-z\sin^{2}v)^{-a} E_{\alpha, \beta}^{\gamma}\left(-p\sec^{2}v\csc^{2}v\right)du,&
\end{align}
\begin{align} \nonumber
{}^{M\!L\!\!}\textrm{F}_{p}(a,b;c;z)= \frac{2}{B(b,c-b)} \int_{0}^{\infty}\sinh^{2b-1}v&\cosh^{2a-2c+1}v(\cosh^{2}v-z\sinh^{2}v)^{-a}& \\	&\times E_{\alpha, \beta}^{\gamma}\left(-p\cosh^{2}v\coth^{2}v\right)du.&
\end{align}
\end{theorem}
\begin{proof}
Direct calculations yield
\begin{align} \nonumber
&{}^{M\!L\!}\textrm{F}_{p}(a,b;c;z)= \sum_{n=0}^{\infty} (a)_n \frac{{}^{M\!L\!\!}B_{p}(b+n,\;c-b)}{B(b,\;c-b)}  \frac{z^n}{n!} & \\ \nonumber &=\frac{1}{B(b,c-b)} \sum_{n=0}^{\infty} (a)_n \int_{0}^{1}t^{b+n-1}(1-t)^{c-b-1}	E_{\alpha, \beta}^{\gamma}\left(-\frac{p}{t(1-t)}\right)\frac{z^n}{n!}dt& \\ \nonumber
&=\frac{1}{B(b,c-b)}  \int_{0}^{1}t^{b-1}(1-t)^{c-b-1}	E_{\alpha, \beta}^{\gamma}\left(-\frac{p}{t(1-t)}\right)\sum_{n=0}^{\infty} (a)_n\frac{{(zt)}^n}{n!}dt& \\ \nonumber
&=\frac{1}{B(b,c-b)}  \int_{0}^{1}t^{b-1}(1-t)^{c-b-1}	E_{\alpha, \beta}^{\gamma}\left(-\frac{p}{t(1-t)}\right)(1-zt)^{-a}dt.&
\end{align}
Setting $u=\frac{t}{1-t}$ in \eqref{14}, we get 
\begin{align}  \nonumber 
{}^{M\!L\!}\textrm{F}_{p}(a,b;c;z)= \frac{1}{B(b,c-b)} \int_{0}^{\infty}u^{b-1}(&1+u)^{a-c}(1+u(1-z))^{-a}E_{\alpha, \beta}^{\gamma}\left(-p\left(2+u+\frac{1}{u}\right)\right)du.& 
\end{align}
Substituting $t=\sin^{2}v$ in \eqref{14}, we have 
\begin{align} \nonumber
{}^{M\!L\!}\textrm{F}_{p}(a,b;c;z)= \frac{2}{B(b,c-b)} \int_{0}^{\frac{\pi}{2}}\sin^{2b-1}v&\cos^{2c-2b-1}v(1-z\sin^{2}v))^{-a} E_{\alpha, \beta}^{\gamma}\left(-p\sec^{2}v\csc^{2}v\right)du.&
\end{align}
On the other hand, substituting $t=\tanh^{2}v$ in \eqref{14}, we get
\begin{align} \nonumber
{}^{M\!L\!}\textrm{F}_{p}(a,b;c;z)= \frac{2}{B(b,c-b)} \int_{0}^{\infty}\sinh^{2b-1}v&\cosh^{2a-2c+1}v(\cosh^{2}v-z\sinh^{2}v)^{-a}& \\  \nonumber 	& \times E_{\alpha, \beta}^{\gamma}\left(-p\cosh^{2}v\coth^{2}v\right)du&
\end{align}
which completes the proof. 
\end{proof}
\begin{theorem} For the ML-generalized CHF, we have the following integral representations:
\begin{align} \label{15} \nonumber
{}^{M\!L\!}{\Phi}_{p}(b;c;z)=&\frac{1}{B(b,c-b)} \int_{0}^{1}t^{b-1}(1-t)^{c-b-1}e^{zt}	E_{\alpha, \beta}^{\gamma}\left(-\frac{p}{t(1-t)}\right)dt,& \\  {}^{M\!L\!}{\Phi}_{p}(b;c;z)=&\frac{1}{B(b,c-b)} \int_{0}^{1}t^{c-b-1}(1-t)^{b-1}e^{z(1-t)}	E_{\alpha, \beta}^{\gamma}\left(-\frac{p}{t(1-t)}\right)dt.&
\\ \nonumber &(p\ge0;\; Re(c)>Re(b)>0,\ Re(\alpha)>0)& 
\end{align}
\end{theorem}
\begin{proof}
A similar procedure yields an integral representations of the ML-generalized CHF by using the definition of the ML-generalized beta function \eqref{4}.
\end{proof}

The differentiation formulas for the ML-generalized GHF and CHF can be obtained by differentiating \eqref{12} and \eqref{13} with respect to the variable $z$ in terms of a parameter by using the formulas:
\begin{eqnarray} \label{16}
B(b,\; c-b)= \frac{c}{b}B(b+1,\; c-b)
\end{eqnarray}
and
\begin{eqnarray}
(a)_{n+1}=a(a+1)_n.
\end{eqnarray}
\begin{theorem}
For the ML-generalized GHF, we have the following differentiation formula:
\begin{eqnarray} \label{17}
\frac{d^n}{dz^n}\left[{}^{M\!L\!\!}\textrm{F}_{p}(a,b;c;z) \right] = \frac{(b)_n(a)_n}{(c)_n} {}^{M\!L\!\!}\textrm{F}_{p}(a+n,b+n;c+n;z).
\end{eqnarray}
\end{theorem}
\begin{proof}
Taking the derivative of ${}^{M\!L\!}\textrm{F}_{p}(a,b;c;z)$ with respect to $z$, we obtain 
\begin{align*} \nonumber
\frac{d}{dz}\left[{}^{M\!L\!}\textrm{F}_{p}(a,b;c;z) \right] &= \frac{d}{dz} \left[\sum_{n=0}^{\infty} (a)_n \frac{{}^{M\!L\!\!}B_{p}(b+n,\;c-b)}{B(b,\;c-b)}  \frac{z^n}{n!}  \right]& \\ &= \sum_{n=1}^{\infty} (a)_n \frac{{}^{M\!L\!\!}B_{p}(b+n,\;c-b)}{B(b,\;c-b)}  \frac{z^{n-1}}{(n-1)!}.&
\end{align*} 
Replacing $n\to n+1$, we get
\begin{align*} \nonumber
\frac{d}{dz}\left[{}^{M\!L\!}\textrm{F}_{p}(a,b;c;z) \right] &= \frac{ba}{c} \left[\sum_{n=0}^{\infty} (a+1)_n \frac{{}^{M\!L\!\!}B_{p}(b+n+1,\;c-b)}{B(b+1,\;c-b)}  \frac{z^n}{n!}  \right]& \\ &= \frac{ba}{c}{}^{M\!L\!}\textrm{F}_{p}(a+1,b+1;c+1;z).&
\end{align*} 
Recursive application of this procedure gives us the general form:
\begin{eqnarray*}
\frac{d^n}{dz^n}\left[{}^{M\!L\!}\textrm{F}_{p}(a,b;c;z) \right] = \frac{(b)_n(a)_n}{(c)_n} {}^{M\!L\!}\textrm{F}_{p}(a+n,b+n;c+n;z)
\end{eqnarray*}
which completes the proof.
\end{proof}
\begin{theorem}
For the ML-generalized CHF, we have the following differentiation formula:
\begin{eqnarray} \label{18}
\frac{d^n}{dz^n}\left[{}^{M\!L\!}{\Phi}_{p}(b;c;z) \right] = \frac{(b)_n}{(c)_n} {}^{M\!L\!}{\Phi}_{p}(b+n;c+n;z).
\end{eqnarray}
\end{theorem}
\begin{proof}
A similar procedure as \eqref{17} gives the result.
\end{proof} 
\begin{theorem}
For the ML-generalized GHF, we have the following Mellin transform:
\begin{align} \label{19}
\mathfrak{M} \left[{}^{M\!L\!\!}\textrm{F}_{p}(a,b;c;z):s\right]= \frac{{}^{M\!L\!}\Gamma(s)B(b+s,c+s-b)}{B(b,\ c-b)} {}_{2}\textrm{F}_{1}(a,b+s;c+2s;z).
\end{align}
\end{theorem}
\begin{proof}
To obtain the Mellin transform, multiply both sides of \eqref{14} by $p^{s-1}$ and integrate with respect to $p$ over the interval $[0,\infty).$Thus we get 
\begin{align} \nonumber
\mathfrak{M} [{}^{M\!L\!}\textrm{F}_{p}(a,b;c;z):s&]= \int_{0}^{\infty} p^{s-1}{}^{M\!L\!}\textrm{F}_{p}(a,b;c;z) dp & \\ \nonumber &= \frac{1}{B(b,c-b)}\int_{0}^{1}t^{b-1}(1-t)^{c-b-1}(1-zt)^{-a} \left[\int_{0}^{\infty} p^{s-1}E_{\alpha, \beta}^{\gamma}\left(-\frac{p}{t(1-t)}dp\right)\right]dt.&
\end{align}
Substituting $u=\frac{p}{t(1-t)}$ in the integral leads to
\begin{align} \nonumber
\int_{0}^{\infty} p^{s-1}	E_{\alpha, \beta}^{\gamma}\left(-\frac{p}{t(1-t)}dp\right)&=\int_{0}^{\infty} u^{s-1}t^s(1-t)^s	E_{\alpha, \beta}^{\gamma}\left(-u\right)du& \\ \nonumber&= t^s(1-t)^s \int_{0}^{\infty} u^{s-1}E_{\alpha, \beta}^{\gamma}\left(-u\right)du& \\ \nonumber&= {}^{M\!L\!}\Gamma(s) t^s(1-t)^s.&
\end{align}
Thus we get 
\begin{align} \nonumber
\mathfrak{M} \left[{}^{M\!L\!}\textrm{F}_{p}(a,b;c;z):s\right]&=\frac{1}{B(b,c-b)}\int_{0}^{1}t^{b+s-1}(1-t)^{c+s-b-1}(1-zt)^{-a} {}^{M\!L\!}\Gamma(s) dt& \\ \nonumber
&=\frac{{}^{M\!L\!}\Gamma(s)}{B(b,c-b)}\int_{0}^{1}t^{b+s-1}(1-t)^{c+2s-(b+s)-1}(1-zt)^{-a} dt& \\ \nonumber &=\frac{{}^{M\!L\!}\Gamma(s)B(b+s,c+s-b)}{B(b,\; c-b)} {}_{2}\textrm{F}_{1}(a,b+s;c+2s;z)&
\end{align}
which completes the proof.
\end{proof}
\begin{coro}
By the Mellin inversion formula, we have the following complex integral representation for ML-generalized GHF: 
\begin{align} \nonumber
&{}^{M\!L\!\!}\textrm{F}_{p}(a,b;c;z)& \\ & = \frac{1}{2\pi i} \int_{-i\infty}^{-i\infty}\frac{{}^{M\!L\!}\Gamma(s)B(b+s,c+s-b)}{B(b,\; c-b)} {}_{2}\textrm{F}_{1}(a,b+s;c+2s;z)p^{-s}ds.&
\end{align}
 \end{coro}
\begin{proof}
Taking Mellin inversion of \eqref{19}, we get the result.
\end{proof}
\begin{theorem}
	For the ML-generalized CHF, we have the following Mellin transform:
	\begin{align} \label{20}
	\mathfrak{M} \left[{}^{M\!L\!}{\Phi}_{p}(b;c;z):s\right]= \frac{{}^{M\!L\!}\Gamma(s)B(b+s,c+s-b)}{B(b,\; c-b)}\Phi(b+s;c+2s;z).
	\end{align}
\end{theorem}
\begin{proof}
A similar argument as \eqref{19} gives the Mellin transform of ML-generalized CHF.
\end{proof}
\begin{coro}
	By the Mellin inversion formula, we have the following complex integral representation for ML-generalized CHF:
	\begin{align} \nonumber
	&{}^{M\!L\!}{\Phi}_{p}(b;c;z)& \\ & = \frac{1}{2\pi i} \int_{-i\infty}^{-i\infty}\frac{{}^{M\!L\!}\Gamma(s)B(b+s,c+s-b)}{B(b,\; c-b)} \Phi(b+s;c+2s;z)p^{-s}ds.&
\end{align}
\end{coro}
\begin{proof}
Taking Mellin inversion of \eqref{20}, we get result.
\end{proof}
\begin{theorem} For the ML-generalized GHF, we have the following transformation formula:
\begin{align} \label{21}
{}^{M\!L\!\!}\textrm{F}_{p}(a,b;c;z)=(1-z)^{-a}{}^{M\!L\!\!}\textrm{F}_{p}\left(a,c-b;b;\frac{z}{z-1}\right),\ (|arg(1-z)|<\pi).
\end{align}
\end{theorem}
\begin{proof}
By writing 
\begin{eqnarray} \nonumber
[1-z(1-t)]^{-a}=(1-z){-a}\left(1+ \frac{z}{1-z}t\right)^{-a}
\end{eqnarray}
and replacing $t\to 1-t$ in \eqref{14}, we obtain 
\begin{align}  \nonumber
&{}^{M\!L\!}\textrm{F}_{p}(a,b;c;z)&
\\ \nonumber &=\frac{(1-z)^{-a}}{B(b,c-b)} \int_{0}^{1}(1-t)^{b-1}t^{c-b-1}\left(1-\frac{z}{z-1}t\right)^{-a}E_{\alpha, \beta}^{\gamma}\left(-\frac{p}{t(1-t)}\right)dt.& 
\end{align}
Hence, 
\begin{align} \nonumber
{}^{M\!L\!}\textrm{F}_{p}(a,b;c;z)=(1-z)^{-a}{}^{M\!L\!}\textrm{F}_{p}\left(a,c-b;b;\frac{z}{z-1}\right)
\end{align}
which completes the proof.
\end{proof}
\begin{rem}
Note that, replacing $z$ by $1-\frac{1}{z}$ in \eqref{21}, we get the following transformation formula, 
\begin{align} 
{}^{M\!L\!\!}\textrm{F}_{p}\left(a,b;c;1-\frac{1}{z}\right)=(z)^{a}{}^{M\!L\!\!}\textrm{F}_{p}\left(a,c-b;b;1-z\right),\ (|arg(z)|<\pi).
\end{align}
Furhermore, replacing $z$ by  $ \frac{z}{1+z}$ in \eqref{21}, we get the following transformation formula, 
\begin{align} 
{}^{M\!L\!\!}\textrm{F}_{p}\left(a,b;c;\frac{z}{1+z}\right)=(1+z)^{a}{}^{M\!L\!\!}\textrm{F}_{p}\left(a,c-b;b;z\right),\ (|arg(1+z)|<\pi).
\end{align}
\end{rem}
\begin{theorem} For the ML-generalized CHF, we have the following transformation formula:
\begin{eqnarray} 
{}^{M\!L\!}{\Phi}_{p}(b;c;z) = \exp(z){}^{M\!L\!}{\Phi}_{p}(c-b;c;-z).
\end{eqnarray} 
\end{theorem}
\begin{proof} Using the definition of ML-generalized CHF \eqref{15}, we have 
\begin{align}  \nonumber
&{}^{M\!L\!}{\Phi}_{p}(b;c;z)=\frac{1}{B(b,c-b)} \int_{0}^{1}t^{b-1}(1-t)^{c-b-1}e^{zt}	E_{\alpha, \beta}^{\gamma}\left(-\frac{p}{t(1-t)}\right)dt,& 
\end{align}	
replacing $t\to 1-t$, we get result.
\end{proof}
\begin{theorem} The recurence relation of the ML-generalized GHF is
\begin{eqnarray}
\Delta_{a}{}^{M\!L\!\!}\textrm{F}_{p}(a,b;c;z)=\frac{bz}{c}{}^{M\!L\!\!}\textrm{F}_{p}(a+1,b+1;c+1;z).
\end{eqnarray}
\end{theorem}
\begin{proof}
By integral \eqref{14} of the ML-generalized GHF and if $\Delta_{a}$ is the shift operator with respect to $a$, we see that
\begin{align} \label{22}
\Delta_{a}&{}^{M\!L\!}\textrm{F}_{p}(a,b;c;z)& \\ 
\nonumber &={}^{M\!L\!}\textrm{F}_{p}(a+1,b;c;z)-{}^{M\!L\!}\textrm{F}_{p}(a,b;c;z)& \\ 
\nonumber &=\frac{1}{B(b,c-b)} \int_{0}^{1}t^{b-1}(1-t)^{c-b-1}(1-zt)^{-a-1}\{1-(1-zt)\} E_{\alpha, \beta}^{\gamma}\left(-\frac{p}{t(1-t)}\right)dt& \\ 
\nonumber &= \frac{z}{B(b,c-b)} \int_{0}^{1}t^{b}(1-t)^{c-b-1}(1-zt)^{-a-1} E_{\alpha, \beta}^{\gamma}\left(-\frac{p}{t(1-t)}\right)dt.& 
\end{align}
Changing $a$, $b$ and $c$ to $a+1$, $b+1$ and $c+1$ in \eqref{14}, respectively,
\begin{align} \label{23}
{}^{M\!L\!}\textrm{F}_{p}(a+1,b+1;c+1;z) 
=\frac{1}{B(b+1,c-b)} \int_{0}^{1}t^{b}(1-t)^{c-b-1}(1-zt)^{-a-1} E_{\alpha, \beta}^{\gamma}\left(-\frac{p}{t(1-t)}\right)dt.
\end{align}
Using the result \eqref{16} and \eqref{23} in \eqref{22}, we have
\begin{eqnarray} \nonumber
\Delta_{a}{}^{M\!L\!}\textrm{F}_{p}(a,b;c;z)=\frac{bz}{c}{}^{M\!L\!}\textrm{F}_{p}(a+1,b+1;c+1;z)
\end{eqnarray} 
which completes the proof. 
\end{proof}
\begin{rem} Setting $z=1$ in \eqref{14}, we have the following relation between defined ML-generalized hypergeometric and ML-generalized beta functions: 
\begin{align}  \nonumber
{}^{M\!L\!}\textrm{F}_{p}(a,b;c;1)=&\frac{1}{B(b,c-b)} \int_{0}^{1}t^{b-1}(1-t)^{c-a-b-1}	E_{\alpha, \beta}^{\gamma}\left(-\frac{p}{t(1-t)}\right)dt& \\ \nonumber =& \frac{{}^{M\!L\!\!}B_{p}(b,\;c-a-b)}{B(b,\;c-b)}&
\end{align}
\end{rem}
\begin{conc}
In these sections 2 and 3, when $\alpha=\beta=\gamma=1$, in integral representations, summation relations, functional relation, Mellin transforms, differentiation formulas, transformation formulas and recurence relations obtained usig ML- generalizations of gamma, beta and hypergeometric functions recudes to  which defined in Chaudry et. al. \cite{3,7}. Also, when $p=0$ and $\beta=1$ (or $\beta=2$) , in these results recudes to classical results for Euler's beta function \cite{1}. 
\end{conc}


\begin{thebibliography}{99}    
	    \bibitem{1}  G.E. Andrews - R. Askey. - R. Roy, \emph{Special Functions}, Cambridge University Press, Camb\-ridge, 1999.
		\bibitem{2} \label{k2} M.A. Chaudry - S.M. Zubair, \emph{Generalized incomplete gamma functions with applications}, J.Comput.Appl.Math. 55 (1994), 99-124. 
		\bibitem{3}  M.A. Chaudry - A. Qadir - M. Rafique - S.M. Zubair, \emph{Extension of Euler's Beta function}, J.Comput.Appl.Math. 78 (1997), 19-32. 
		\bibitem{4}  M.A. Chaudry - S.M. Zubair, \emph{On the decomposition of generalized incomplete gamma functions with applications to Fourier transforms}, J.Comput.Appl.Math. 59 (1995), 253-284. 
		\bibitem{5}  M.A. Chaudry - N.M. Temme - E.J.M. Veling, \emph{Asymptotic and closed form of a generalized incomplete gamma functions}, J.Comput.Appl.Math. 67 (1996), 371-379.
		\bibitem{6} \label{k6} M.A. Chaudry - S.M. Zubair, \emph{Extended incomplate gamma functions with applications}, J.Math.Anal.Appl. 274 (2002), 725-745.
	    \bibitem{7}  M.A. Chaudry - A. Qadir - H.M. Srivastava, - R.B. Paris,  \emph{Extended hypergeometric and confluent hypergeometric functions}, Appl. Math. Comput. 159 (2004), 589-602.    
	    \bibitem{8} P. J. Konovska, \emph{Inequalities and Asymptotic Formulae for the Three Parametric Mittag-Leffer Function}, Math. Balkanica, 26 (2012), 203-210.
		\bibitem{9}  A.M. Mathai, H.J. Haubold, \emph{Special Functions for Applied Scients}, 1st edition (Springer, New York, 2008) 	
	    \bibitem{10} \label{k10}  A.R. Miller, \emph{Reduction of a generalized incomplate gamma function, related Kampe de Feriet functions, and incomplate Weber integrals}, Rocky Mountain J. Math. 30 (2000), 703-714.
	    \bibitem{11} T.R. Prabhakar, \emph{A singuler integral equation with a generalized Mittag-Leffer function in the kernel}, Yokohoma Math. J. 19 (1971), 7-15.
	    
	    
	\end{thebibliography}
\end{document}